\documentclass[12pt,reqno]{amsart}
\usepackage[widespace]{fourier}
\usepackage{hyperref}
\hypersetup{colorlinks=true,linkcolor=blue, citecolor=blue}
\usepackage{amsmath,amscd,amsfonts,amssymb}
\usepackage[all]{xypic}
\usepackage[utf8]{inputenc}
\usepackage{longtable}
\numberwithin{equation}{section}
\newtheorem{theorem}{Theorem}[section]
\newtheorem{definition}[theorem]{Definition}

\newtheorem{lemma}[theorem]{Lemma}
\newtheorem{proposition}[theorem]{Proposition}

\newtheorem{remark}[theorem]{Remark}

\begin{document}

\title{On the Rationality of Nagaraj-Seshadri Moduli Space}

\author[P. Barik]{Pabitra Barik}

\address{Department of Mathematics, Indian Institute of Technology-Madras, Chennai, India }

\email{pabitra.barik@gmail.com}

\author[A. Dey]{Arijit Dey}

\address{Department of Mathematics, Indian Institute of Technology-Madras, Chennai, India }

\email{arijitdey@gmail.com}

\author[Suhas, B N]{Suhas, B N}
\address{Department of Mathematics, Indian Institute of Technology-Madras, Chennai, India }
\email{chuchabn@gmail.com}

\maketitle
\abstract
We show that each of the irreducible components of moduli of rank $2$ torsion-free sheaves with odd Euler characteristic over a reducible nodal curve is rational.
\endabstract
\section{Introduction}

Let $X$ be a reducible nodal curve over an algebraically closed field $k$ of characteristic 0 such that it is a union of two smooth irreducible components $X_{1}$ of genus $g_{1} \geq 2$ and $X_{2}$ of genus $g_{2} \geq 2$ meeting exactly at one node $p$. Let ${\bf a}\,=\,(a_{1},a_{2})$ be a tuple of positive rational numbers such that $a_1+a_2\,=\,1;$ we call this a polarisation on $X$. Let $\chi$ be an integer such that $a_{1}\chi$ is not an integer. In this setting it is a theorem of
Nagaraj-Seshadri \cite[Theorem 4.1]{N-S} that the moduli space $M(2,{\bf{a}},\chi)$ of semi-stable rank two torsion-free sheaves on $X$ with Euler characteristic $\chi$ is a reduced, connected projective scheme with exactly two irreducible components, and when $\chi$ is odd, the moduli space is a union of two smooth varieties $M_{12}$ and $M_{21}$ intersecting transversally along a smooth divisor $N$.

Let ${\bf \xi}\,= \,(L_{1},L_2)$, where $L_1$ and $L_2$ are two invertible sheaves on $X_{1}$ and $X_{2}$ (of suitable \text{\text{deg}}rees) respectively.
Then in \cite[Section 7]{N-S} the analogue of a "fixed determinant moduli space"  has been defined and we denote it by $M(2,{\bf a},\chi,{\bf \xi})$. It is shown in (\cite{Hua}, \cite{S-B})  that when $\chi$ is odd and $a_1\chi$ is not an integer, $M(2,{\bf a},\chi,{\bf \xi})$ is also a reduced, connected projective scheme with exactly two smooth components meeting transversally along a smooth divisor. The main result of this article is the following:

\begin{theorem}\label{Theorem 1.1}
 If $gcd~(\chi,2)~=~1,$ then both the irreducible components of $M(2,{\bf a},\chi,{\bf \xi})$ are rational. In particular $M(2,{\bf a},\chi,{\bf \xi})$ is rationally connected.
\end{theorem}
Over a smooth projective curve of genus $g \geq 2,$ the rationality of the moduli space was first proved by Tjurin \cite[Theorem 14]{TJ} in the rank $2$ and odd \text{\text{deg}}ree case. When rank and \text{\text{deg}}ree are coprime this result was generalized by Newstead \cite{N}, \cite{N2}, King and Schofield  \cite{K-S} in higher order of generalities.
It is still not known if the moduli space is rational or not in the non-coprime case, even for rank $2$ and \text{\text{deg}}ree $0$. In the non-smooth case, when the curve is irreducible and has any number of nodal singularities and genus $\geq2,$ rationality in the coprime case was proved by Bhosle and Biswas \cite[Theorem 3.7]{B-B}.  Over a reducible nodal curve $X$ as described above it has been shown by Basu that each irreducible component of $M(2,{\bf a},\chi,{\bf\xi})$ is unirational \cite[Lemma 2.5]{S-B}. Motivated by this result we go to the next step i.e. to prove rationality of each of these components. The proof of our result broadly follows the strategy of Newstead \cite{N} but involves several technical difficulties.

It is well known that the moduli space of bundles over curves has a good specialization property, i.e. if a smooth projective curve $Y$ specializes to a projective curve $X$ with nodes as the only singularities, then the moduli space of vector bundles $M_Y$ on $Y$ specializes to the moduli space of torsion-free sheaves $M_X$ on $X$ \cite{Gi}, \cite{N-S-2}, \cite{S1}, \cite{S2}. It is known that rationality of projective varieties does not have a good specialization property, for example a family of cubic surface which is rational specializes to a non-rational surface which is birational to $E \times \mathbb P^1$ where $E$ is a cubic curve . 
Our result shows that in the rank $2$ and odd Euler characteristic case the moduli space of vector bundles gives an example of a family of rational varieties specializing to a rationally connected variety with two irreducible rational components. We hope that Theorem \ref{Theorem 1.1} will be useful in the study of \text{\text{deg}}eneration of higher \text{\text{\text{\text{dim}}}}ensional smooth projective algebraic varieties.

Further it will be interesting to see if Theorem \ref{Theorem 1.1} can be generalized to a more general situation i.e if the underlying curve $C$ has more than $2$ components together with more than one node. In such a general situation the moduli space of semistable torsion-free sheaves (arbitrary rank) has been constructed by Seshadri (see [Chapter VII, \cite{S2}). In particular when $C$ is a tree like curve without any rational components, then the number of components of the moduli space and inequalities involving Euler characteristics has been computed by Montserrat Teixidor I Bigas \cite{tei}, \cite[Theorem 3.2]{tei2}. It will be interesting to investigate the rationality of each of these components.

{\bf Acknowledgement.} It is a pleasure to thank V. Balaji and D.S. Nagaraj for having many useful discussions during
the period of this work. We thank P. E. Newstead  for his valuable comments and remarks. We also thank Suratno Basu for answering some of our questions which helped us in understanding Nagaraj-Seshadri's paper. We would like to thank the referee for pointing out a gap in the older version of our manuscript. 

\section{A Brief Description of the Moduli Space}

In this section, we shall briefly recall some of the results proved in \cite{N-S} which will be useful in later sections. Let $X$  be a reducible projective nodal curve as before which has two smooth irreducible components $X_1$ and $X_2$ meeting at the nodal point $p$.  Any torsion-free sheaf $E$ on $X$ can be identified with a triple $(E_{1}, E_{2},\overrightarrow{T})$ or $( E_{1}^{\prime}, E_{2}^{\prime} , \overleftarrow{S} )$ where $E_{i}$'s are locally free sheaves over $X_{i}$'s and $\overrightarrow{T}$ and $\overleftarrow{S}$ are linear maps from $E_{1}(p)$ to $E_{2}(p)$ and $E_{2}^{\prime}(p)$ to $E_{1}^{\prime}(p)$ respectively. In fact in \cite[Lemma 2.3]{N-S}, an equivalence is shown between the category of torsion-free sheaves and the category of triples.  Let $E$ be a torsion-free sheaf on $X$ identified by the triple $(E_{1}, E_{2},\overrightarrow{T})$ as well as $( E_{1}^{\prime}, E_{2}^{\prime} , \overleftarrow{S} )$. Then we have the following equality of Euler characteristics between them (see\cite{N-S},Remark 2.11)-
\begin{equation}\label{1}
 \chi(E)= \chi(E_{1}, E_{2},\overrightarrow{T})= \chi( E_{1}^{\prime}, E_{2}^{\prime} , \overleftarrow{S}),
\end{equation}

 where
 \begin{eqnarray}\label{2}
  \chi(E_{1}, E_{2},\overrightarrow{T}):= \chi(E_{1}) + \chi(E_{2})- rk(E_{2}),
 \end{eqnarray}
  and
 \begin{equation}\label{3}
  \chi( E_{1}^{\prime}, E_{2}^{\prime} , \overleftarrow{S}):= \chi( E_{1}^{\prime})+ \chi(E_{2}^{\prime}) - rk(E_{1}^{\prime}).
 \end{equation}

 Let ${\bf a}\,=\,(a_{1},a_{2})$ be a polarisation on $X$ with $a_i \, > \,0$ rational numbers and $a_{1} + a_{2} = 1.$ For every non-zero triple $(E_{1},E_{2},\overrightarrow{T}),$ we define
  \begin{equation*}
   \mu((E_{1},E_{2},\overrightarrow{T}))\, = \,\frac{\chi(E_{1},E_{2},\overrightarrow{T})}{a_{1}rk(E_{1})+a_{2}rk(E_{2})}.
  \end{equation*}
  \begin{definition}
  Let $(E_{1},E_{2},\overrightarrow{T})$ be a triple. We say that the triple $(E_{1},E_{2},\overrightarrow{T})$ is stable (resp. semi-stable) if for every proper subtriple $(G_{1},G_{2},\overrightarrow{U}),$ \\
  $\mu((G_{1},G_{2},\overrightarrow{U})) \,<\, \mu((E_{1},E_{2},\overrightarrow{T}))$ (resp. $\leq$).
  \end{definition}

When $\chi$ is odd and $a_1\chi$ is not an integer the moduli space $M(2,{\bf a},\chi)$ of semistable torsion-free sheaves on $X$ with Euler characteristic $\chi$ is a reduced, connected projective variety  with the two smooth irreducible components $M_{12}$ and $M_{21}$ intersecting transversally on a smooth projective variety $N$ \cite[Theorem 4.1]{N-S}. The irreducible components $M_{12}$ and $M_{21}$ have the following description in terms of triples:

 The first component $M_{12}$ is a smooth projective variety which is a fine moduli space of stable triples $(E_{1}, E_{2},\overrightarrow{T})$ such that $E_{i}$'s are rank 2 vector bundles over $X_{i}$'s, and $\overrightarrow{T}:\,E_{1}(p)\rightarrow E_{2}(p)$ is a nonzero linear map such that
 \begin{equation}\label{4}
  a_{1}\chi \,< \,\chi(E_{1}) \,< \,a_{1}\chi +1 ,~a_{2}\chi +1 \,< \,\chi(E_{2}) \,< \,a_{2}\chi +2 .
 \end{equation}
The second component $M_{21}$ also has a similar description in terms of triples. It is a smooth projective variety which is a fine moduli space of stable triples $( E_{1}^{\prime}, E_{2}^{\prime} , \overleftarrow{S})$ such that $E_{i}^{\prime}$'s are rank 2 vector bundles over $X_{i}$'s, and $\overleftarrow{S}:\, E_{2}^{\prime}(p) \rightarrow E_{1}^{\prime}(p)$ is a nonzero linear map such that
 \begin{equation}\label{5}
  a_{1}\chi +1 < \chi(E_{1}^{\prime}) < a_{1}\chi +2 ,~a_{2}\chi < \chi(E_{2}^{\prime}) < a_{2}\chi +1.
 \end{equation}
The intersection $N = M_{12} \cap M_{21}$ can be identified with $P_{1}\times P_{2}$ where $P_{i}$'s are certain parabolic moduli spaces over $X_{i}$'s (see \cite{N-S}, Theorem 6.1 for details).
In terms of triples, $N$ is given by
\begin{eqnarray*}
   \{[E_{1},E_{2},\overrightarrow{T}] \in M_{12}~|~rk(\overrightarrow{T})=1\},
\end{eqnarray*}
which can be identified with
\begin{eqnarray*}
 \{[E_{1}^{\prime},E_{2}^{\prime},\overleftarrow{S}] \in M_{21}~|~rk(\overleftarrow{S})=1\}.
\end{eqnarray*}


In this paper, we are interested in the fixed determinant case. Let $E \in M(2,{\bf a},\chi)$ be identified by the triple $(E_{1},E_{2},\overrightarrow{T})$ as well as the triple $(E_{1}^{\prime},E_{2}^{\prime},\overleftarrow{S}).$ Let $\chi_{i}:= \chi(E_{i})$ and $\chi_{i}^{\prime}:= \chi(E_{i}^{\prime}),$ for $i=1,2.$ If $\chi_{i}$ satisfy the inequalities in \eqref{4}, then $E \in M_{12}$ and if $\chi_{i}^{\prime}$ satisfy the inequalities in \eqref{5}, then $E \in M_{21}.$ Let $J^{d_{i}}(X_{i})$ be the Jacobian of line bundles of \text{\text{deg}}ree $d_{i}$ on $X_{i},$ where $d_{i}= \chi_{i}-2(1-g_{i}),$ for $i=1,2.$ Now by \cite[Proposition 7.1]{N-S}, there is a well-defined surjective morphism
\begin{eqnarray*}
 \text{det}\,:M(2,{\bf a},\chi \neq 0) & \longrightarrow & J^{d_{1}}(X_{1}) \times J^{d_{2}}(X_{2}),
\end{eqnarray*}
given by
\begin{eqnarray*}
E      & \mapsto &  (\Lambda^{2}(E_{1}),\Lambda^{2}(E_{2})) \,\,\hspace{.3cm} \text{if}  \,\,\hspace{.3cm} E \in M_{12}, ~\,\,\,and \\
E      & \mapsto &  \Phi((\Lambda^{2}(E_{1}^{\prime})),\Lambda^{2}(E_{2}^{\prime})) \,\,\hspace{.3cm} \text{if} \,\,\hspace{.3cm} E \in M_{21},
\end{eqnarray*}
where
\begin{equation*}
 \Phi:J^{d_{1}+1}(X_{1}) \times J^{d_{2}-1}(X_{2})  \rightarrow  J^{d_{1}}(X_{1}) \times J^{d_{2}}(X_{2})
\end{equation*}
is an isomorphism defined by
\begin{equation*}
 (L_{1},L_{2})  \mapsto  (L_{1} \otimes \mathcal{O}_{X_{1}}(-p),L_{2} \otimes \mathcal{O}_{X_{2}}(p)).
\end{equation*}
Now let us fix $L_{1} \in J^{d_{1}}(X_{1})$ and $L_{2} \in J^{d_{2}}(X_{2})$, we write ${\bf \xi}\,=\,(L_1,L_2)$. Then the fixed determinant moduli space $M(2,\mathbf{a},\chi,\xi)$ is by definition $\text{det}\,^{-1}({\bf \xi}).$ By \cite[Proposition 7.2]{N-S}, it is reduced.
Let $\text{det}\,_{12}$ (resp. $\text{det}\,_{21}$) be the morphism $\text{det}\,|_{M_{12}}$ (resp. $\text{det}\,|_{M_{21}}$). For notational convenience we write $M_{12}({\bf \xi})$ (resp. $M_{21}({\bf \xi})$) for $\text{det}\,_{12}^{-1}({\bf \xi})$ (resp. $\text{det}\,_{21}^{-1}({\bf \xi})$).
Then we have
\begin{eqnarray*}
 M_{12}({\bf \xi})  & = & \{[E_{1},E_{2},\overrightarrow{T}] \in M_{12} ~|~ \Lambda^{2}(E_{1})=L_{1},~\Lambda^{2}(E_{2})=L_{2} \},
\end{eqnarray*}
and
\begin{eqnarray*}
M_{21}({\bf \xi})   & = & \{[E_{1}^{\prime},E_{2}^{\prime},\overleftarrow{S}] \in M_{21} ~|~ \Lambda^{2}(E_{1}^{\prime})=L_{1} \otimes \mathcal{O}_{X_{1}}(p),\Lambda^{2}(E_{2}^{\prime})=L_{2} \otimes \mathcal{O}_{X_{2}}(-p)) \}.
\end{eqnarray*}
By \cite[Proposition 6.5]{S-B}, the fixed determinant moduli space is a connected, projective scheme with exactly two smooth irreducible components $M_{12}({\bf \xi})$ and $M_{21}({\bf \xi})$, meeting transversally along the smooth divisor $N({\bf \xi}) = M_{12}({\bf \xi}) \cap N$ (which is identified with $M_{21}({\bf \xi}) \cap N$). Since $\chi$ is assumed to be an odd integer, and $\chi = \chi_{1} + \chi_{2}-2,$ we can conclude that either $\chi_{1}$ is odd or $\chi_{2}$ is odd and not both. (Same argument applies to $\chi_{1}^{\prime}$ and $\chi_{2}^{\prime}$ also).
\par
Our aim in this paper is to prove that both $M_{12}({\bf \xi})$ and $M_{21}({\bf \xi})$ are rational. First we prove that $M_{12}({\bf \xi})$ is rational. That the other component $M_{21}({\bf \xi})$ is rational follows from similar arguments.

\section{Construction of a stable family}

Let $i_{1}$ and $i_{2}$ be the closed immersions given by $X_{1} \rightarrow X$ and $X_{2} \rightarrow X$ respectively. We choose an invertible sheaf $L_{1}$ on $X_{1}$ such that it is generated by global sections and is of \text{\text{deg}}ree $2g_{1}-1$ (see Remark \ref{Remark 3.2}(a)). Let $L_{2}$ be an invertible sheaf on $X_{2}$ of \text{\text{deg}}ree $2g_{2}.$ Clearly $H^{1}(X_{j},L_{j})=0,$ for $j=1,2.$ Also by \cite[Lemma 5.2]{PN}, $L_{2}$ is generated by global sections. Now if $\overrightarrow{\lambda}:L_{1}(p) \rightarrow L_{2}(p)$ is an isomorphism of vector spaces, then the triple $(L_{1},L_{2},\overrightarrow{\lambda})$ corresponds to an invertible sheaf $L$ on $X$ and we have
\begin{eqnarray}\label{6}
\chi(L) = \chi(L_{1},L_{2},\overrightarrow{\lambda}) =  \chi(L_{1}) + \chi(L_{2}) -rk(L_{2}) = g ,
\end{eqnarray}
where $g=g_{1}+g_{2}.$ Also by \cite[Proposition 2.2]{N-S}, we have the following short exact sequence
\begin{equation}\label{7}
 0\rightarrow L \rightarrow i_{1_{*}}(L_{1}) \oplus i_{2_{*}}(L_{2}) \rightarrow T_{\lambda} \rightarrow 0,
\end{equation}
where $T_{\lambda}$ is supported only at $p,$ and over the residue field $k(p),$ it is a vector space of \text{\text{\text{\text{dim}}}}ension one. By \cite[Lemma 2.3]{N-S},
 \begin{equation}\label{8}
 L = \{ (v,w) \in i_{1_{*}}(L_{1}) \oplus i_{2_{*}}(L_{2}) ~|~ \overrightarrow{\lambda}(v(p)) = w(p) \}.
 \end{equation}
\begin{lemma}\label{Lemma 3.1}
Let $L$ be as above. Then

  (i) The functor $H^{0}(X,-)$ applied to \eqref{7} is exact.

  (ii) \text{\text{\text{\text{dim}}}} ($H^{0}(X,L))=g$ and \text{\text{\text{\text{dim}}}} ($H^{1}(X,L))=0.$

  (iii) \text{\text{\text{\text{dim}}}} ($H^{0}(X,L^{*}))=0$ and \text{\text{\text{\text{dim}}}} ($H^{1}(X,L^{*}))= 3g-2,$ where $L^{*}$ is the dual of $L.$

\begin{proof}
 Applying the functor $H^{0}(X,-)$ to \eqref{7}, we get the exact sequence
\begin{equation}\label{9}
 0\rightarrow H^{0}(X,L) \rightarrow H^{0}(X_{1},L_{1}) \oplus H^{0}(X_{2},L_{2}) \xrightarrow{\beta} H^{0}(X,T_{\lambda}).
\end{equation}
(Here we are using the fact that $H^{0}(X_{j},L_{j})=H^{0}(X,i_{j_{*}}(L_{j}))$ as $i_{j}$'s are closed immersions $X_{j} \rightarrow X$).
Our aim is to show that the map  $\beta$ is surjective. Since $H^{0}(X,T_{\lambda})$ is one \text{\text{\text{\text{dim}}}}ensional, it is enough to show that $\beta$ is a non-zero map. Now, as $\text{\text{deg}}~(L_{1})=2g_{1}-1$ and $H^{1}(X_{1},L_{1})=0,$ it follows that $\text{\text{\text{\text{dim}}}}~(H^{0}(X_{1},L_{1}))=g_{1}.$ Similarly $\text{\text{\text{\text{dim}}}}~(H^{0}(X_{2},L_{2}))=g_{2}+1.$

Consider the natural maps
\begin{equation}\label{10}
 \phi_{j}:H^{0}(X_{j},L_{j}) \rightarrow L_{j}(p),
\end{equation}
for $j=1,2.$ As $L_{j}$'s are generated by global sections, these maps are surjective. Let $(v_{1},w_{1}) \in H^{0}(X_{1},L_{1}) \oplus H^{0}(X_{2},L_{2})$ be such that $v_{1}(p) \neq 0,$  $w_{1}(p) \neq 0$ and $\overrightarrow{\lambda}(v_{1}(p))=w_{1}(p).$ Then by \eqref{8}, $(v_{1},0)$ does not belong to $H^{0}(X,L),$ hence $\beta(v_{1},0) \neq 0$ in $H^{0}(X,T_{\lambda}).$ This proves $(i)$.

$(ii)$ is a direct consequence of $(i)$.

Since pull-back operation commutes with tensor product,
we have $i_{j}^{*}(L^{*})= L_{j}^{*},$ for $j=1,2.$ As $\text{\text{deg}}~(L_{1})= 2g_{1}-1$ and $\text{\text{deg}}~(L_{2})=2g_{2},$ we have $H^{0}(X_{j},L_{j}^{*})=0$ for $j=1,2.$ Now by \eqref{7}, \eqref{9} (applying to $L^{*}$) we get
$H^{0}(X,L^{*})=0$. Hence,
\begin{eqnarray*}
 \text{\text{\text{\text{dim}}}}~ (H^{1}(X,L^{*})) &=& - \chi(L^{*}) \\
                    &=& - \chi(L_{1}^{*}) - \chi(L_{2}^{*}) +1 \\
                    &=& 3g-2.
\end{eqnarray*}
This proves $(iii).$
\end{proof}
\end{lemma}

\begin{remark}\label{Remark 3.2}
(a) In the above Lemma, we assume $L_{1}$ on $X_{1}$ to be globally generated and of $\text{\text{deg}}ree~2g_{1}-1.$ To see the existence of such an invertible sheaf $L_{1}$, one can take a \text{\text{deg}}ree one invertible sheaf $\mathcal{L}_{1}$ on $X_{1}$ such that $H^{0}(X_{1},\mathcal{L}_{1})=0$ and define $L_{1}$ to be $\omega_{X_{1}} \otimes \mathcal{L}_{1},$ where $\omega_{X_{1}}$ is the canonical sheaf on $X_{1}.$ The existence of such an $\mathcal{L}_{1}$ is clear because the genus $g_{1} \geq 2.$

(b) Let $q = 3g-2.$ Then by fixing a basis of $H^{1}(X,L^{*}),$ we can identify it with $k^{q}.$ We have the natural $k^{*}-$action on $k^{q}$ and
      \begin{equation*}
      W = \{ (a_{1},a_{2}, \cdots,a_{q}) \in k^{q}~| ~a_{1} \neq 0 \}
      \end{equation*}
   is clearly an invariant open subset of $k^{q}$ under the $k^{*}-$ action.

   Let $A:= \{(a_{1},a_{2}, \cdots,a_{q})\in W ~|~ a_{1}=1 \}$ (Clearly $A$ is Zariski closed  and every orbit of the $k^{*}-action$ on $W$ meets $A$ in exactly one point).
\end{remark}

Since the maps $\phi_{j}$'s mentioned in \eqref{10} are surjective, we have $\text{\text{\text{\text{dim}}}}~(ker(\phi_{1}))=g_{1}-1$ and $\text{\text{\text{\text{dim}}}}~(ker(\phi_{2}))=g_{2}.$ Let $\{v_{2}, \cdots, v_{g_{1}}\}$ be a basis of $ker(\phi_{1})$ and $\{w_{2}, \cdots, w_{g_{2}+1}\}$ be a basis of $ker(\phi_{2}).$ These bases can be extended to the bases $\{v_{1},v_{2}, \cdots, v_{g_{1}}\}$ and $\{w_{1},w_{2}, \cdots, w_{g_{2}+1}\}$ of $H^{0}(X_{1},L_{1})$ and $H^{0}(X_{2},L_{2})$ respectively where $v_{1}$ and $w_{1}$ are as in the proof of the Lemma \ref{Lemma 3.1}. It is also clear from \eqref{8}  that $(v_{1},w_{1}),(v_{2},0),\cdots,(v_{g_{1}},0),
  (0,w_{2}),\cdots,(0,w_{g_{2}+1})$ will form a basis for $H^{0}(X,L).$

Suppose $(0,0) \neq (v,w) \in H^{0}(X,L).$ Then we have
      \begin{eqnarray*}
       (v,w) &=& \alpha_{1}(v_{1},w_{1}) + \alpha_{2}(v_{2},0) + \cdots + \alpha_{g_{1}}(v_{g_{1}},0)\\
        & & + \beta_{2}(0,w_{2}) + \cdots + \beta_{g_{2}+1}(0,w_{g_{2}+1}),
      \end{eqnarray*}
where $\alpha_{i}$'s and $\beta_{j}$'s are scalars and at least one of them is non-zero.

We know that every non-zero section $(v,w)$ defines a non-zero map  $\mathcal{O}_{X} \rightarrow L.$ Further, this map is injective if and only if both $v$ and $w$ are non-zero which is true if and only if at least one $\alpha_{i} \neq 0$ and at least one $\beta_{j} \neq 0$ or $\alpha_{1} \neq 0.$ Let
      \begin{equation*}
        C^{\prime} = \{(v,w):= \phi \in H^{0}(X,L)~|~\phi :\, \mathcal{O}_{X} \hookrightarrow L ~\text{injective}\}.
      \end{equation*}
      Clearly $C^{\prime}$ is a non-empty open subset in $H^{0}(X,L).$

\begin{lemma}(cf.~\cite{H-L})\label{Lemma 3.3}
 Let $L$ be as above. Then there exists a vector space $V$ and a universal extension
 \begin{equation}\label{11}
  0 \rightarrow \mathcal{O}_{X \times V} \rightarrow \tilde{\mathcal{E}} \rightarrow \pi^{*}(L) \rightarrow 0
 \end{equation}
of bundles over $V \times X$ (where $\pi:V \times X \rightarrow X$ is the projection map), such that there is a natural isomorphism
\begin{equation*}
 \alpha:V \rightarrow H^{1}(X,L^{*})
\end{equation*}
where for each $v \in V$, $\alpha(v)$ is the element corresponding to the restriction of the extension \eqref{11} to $\{v\} \times X.$
\end{lemma}

\begin{remark}\label{Remark 3.4}
\begin{enumerate}
\item{}
Suppose $\tilde{\mathcal{E}}$ is as in Lemma \ref{Lemma 3.3} and $v \in H^{1}(X,L^{*})$ is such that $\text{\text{\text{\text{dim}}}}~(H^{0}(X,\tilde{\mathcal{E}}_{v}))=1.$ Then one can easily see that for any $w \in H^{1}(X,L^{*}),$ $\tilde{\mathcal{E}}_{v} \,\cong \, \tilde{\mathcal{E}}_{w}$ if and only if $v$ and $w$ are in the same orbit under the natural action of $k^{*}$ on $H^{1}(X,L^{*}).$
\item{} When $X$ is smooth the above lemma was proved in \cite[Proposition 3.1, pp. 19-20]{N-R}.
\end{enumerate}
\end{remark}

\begin{lemma}\label{Lemma 3.5}
 Let $L_1$ be as above. Then there exists an extension
\begin{equation}\label{12}
 0\rightarrow \mathcal{O}_{X_1} \rightarrow E_1 \rightarrow L_1 \rightarrow 0 ,
\end{equation}
for which $\text{\text{\text{\text{dim}}}}~(H^{0}(X_1,E_1)=1,$ and such an $E_1$ is stable.
\end{lemma}

\begin{proof}
 The existence of such an extension on $X_1$ can be seen as a special case of \cite[Lemma 5]{N}, and stability of the bundle $E_1$ can be seen as a special case of \cite[Lemma 6]{N}.
\end{proof}

\begin{lemma}\label{Lemma 3.6}
 Let $L_2$ be as above. Then there exists an extension
 \begin{equation}\label{13}
 0\rightarrow \mathcal{O}_{X_2} \rightarrow E_2 \rightarrow L_2 \rightarrow 0 ,
\end{equation}
for which $\text{\text{\text{\text{dim}}}}~(H^{0}(X_2,E_2)=2,$ and such an $E_2$ is semi-stable.
\end{lemma}

\begin{proof}
  Suppose $e_2 \in H^{1}(X_2,L_2^{*})$ and \eqref{13} is the corresponding extension. Then it is clear that $\chi(E_2)=2$ and therefore $\text{\text{\text{\text{dim}}}}~(H^{0}(X_2,E_2)) \geq 2.$

  Suppose $\phi_2 \in H^{0}(X_2,L_2)$ is any non-zero section. Then we have an injective morphism
  \begin{eqnarray}\label{14}
    \phi_2:\,\mathcal{O}_{X_2} & \hookrightarrow & L_2 .
  \end{eqnarray}
Tensoring \eqref{14} by the canonical sheaf $\omega_{X_2}$ and applying the global section functor, we get the map
 \begin{equation*}
  H^{0}(X_2,\omega_{X_2}) \hookrightarrow H^{0}(X_2,L_2 \otimes \omega_{X_2}).
 \end{equation*}
Taking dual and using the duality theorem, we get the map
 \begin{equation*}
  H^{1}(X_2,L_2^{*}) \xrightarrow{\tilde{\phi_2}}  H^{1}(X_2,\mathcal{O}_{X_2}) .
 \end{equation*}
 Clearly $\tilde{\phi_2}$ is onto. This implies
 \begin{eqnarray}\label{15}
  \text{\text{\text{\text{dim}}}}~(ker(\tilde{\phi_2})) & = & \text{\text{\text{\text{dim}}}}~(H^{1}(X_2,L_2^{*})) -g_2  ~> ~0.
 \end{eqnarray}
 Applying the sheaf functors $\mathcal{H}om(L_2,-)$ and $\mathcal{H}om(\mathcal{O}_{X_2},-)$ to \eqref{13} and taking the long exact sequence, we get the following
 commutative diagram -
$$\xymatrix{
0 \ar@{->}[r]^{} \ar@{->}[d]^{} & Hom(L_2,\mathcal{O}_{X_2}) \ar@{->}[r]^{} \ar@{->}[d]^{} & Hom(L_2,E_2) \ar@{->}[r]^{} \ar@{->}[d]^{} & Hom(L_2,L_2) \ar@{->}[r]^{} \ar@{->}[d]^{} & H^{1}(X_2,L_2^{*}) \ar@{->}[r]^{}  \ar@{->}[d]^{\tilde{\phi_2}}  & \cdots \\
0 \ar@{->}[r]^{} & H^{0}(X_2,\mathcal{O}_{X_2})  \ar@{->}[r]^{} & H^{0}(X_2,E_2)  \ar@{->}[r]^{} & H^{0}(X_2,L_2) \ar@{->}[r]^{} & H^{1}(X_2,\mathcal{O}_{X_2}) \ar@{->}[r]^{} & \cdots
}$$
From this diagram, it is clear that $\phi_2$ lifts to a section on $E$ if and only if $\tilde{\phi_2}(e_2)=0.$ This fact is proved in \cite[Lemma 3.1]{N-R}, in greater generality. Also
\begin{eqnarray}\label{16}
 \text{\text{\text{\text{dim}}}}~(H^{0}(X_2,E_2)) &=& \text{\text{\text{\text{dim}}}}~(H^{0}(X_2,\mathcal{O}_{X_2})) + \text{\text{\text{\text{dim}}}}~(ker(H^{0}(X_2,L_2)) \rightarrow H^{1}(X_2,\mathcal{O}_{X_2})) \nonumber \\
  & = & 1 + \text{\text{\text{\text{dim}}}}~(ker(H^{0}(X_2,L_2)) \rightarrow H^{1}(X_2,\mathcal{O}_{X_2})).
\end{eqnarray}
So
\begin{eqnarray*}
 \text{\text{\text{\text{dim}}}}~(H^{0}(X_2,E_2))~ >~2 & \Leftrightarrow & \text{\text{\text{\text{dim}}}}~(ker(H^{0}(X_2,L_2)) \rightarrow H^{1}(X_2,\mathcal{O}_{X_2})) > 1.
\end{eqnarray*}
Suppose there exists an $e_2 \in \text{\text{\text{\text{ker}}}}~\tilde{\phi_2}$ (with \eqref{13} as the corresponding extension) such that $\phi_2$ is the only section (up to scalar multiplication) which lifts to $E_2.$ Then we are done.

Suppose this is not the case with any non-zero section $\phi_2 \in H^0(X_2,L_2).$ This means for every non-zero section $\phi_2 \in H^0(X_2,L_2)$ and every $e_2 \in \text{\text{\text{\text{ker}}}}~\tilde{\phi_2}$ (with \eqref{13} as the corresponding extension) there are at least two linearly independent sections that lift to the corresponding bundle $E_2.$

Let
\begin{equation*}
 Y = \{(e_2,\phi_2)~|~\phi_2 \neq 0~and~\tilde{\phi_2}(e_2)=0\} \subset H^1(X_2,L_2^*) \times H^0(X_2,L_2).
\end{equation*}
This implies
\begin{eqnarray*}
 \text{\text{\text{\text{dim}}}}~(Y) &=& \text{\text{\text{\text{dim}}}}~(H^{0}(X_2,L_2))+ \text{\text{\text{\text{dim}}}}~(H^{1}(X_2,L_2^{*})) -g_2 \\
                                     &=& \text{\text{\text{\text{dim}}}}~(H^{1}(X_2,L_2^{*})) + 1.
\end{eqnarray*}
(The last equality is true because $\text{\text{\text{\text{dim}}}}~(H^{0}(X_2,L_2))= g_2+1).$

Now if $e_2 \in p_1(Y)$ (where $p_1$ is the first projection map from $H^1(X_2,L_2^*) \times H^0(X_2,L_2)$), then $\text{\text{\text{\text{dim}}}}~p_1^{-1}(e_2) \cap Y \geq 2$ because $e_2 \in p_1(Y)$ implies the corresponding bundle $E_2$ has at least two linearly independent lifts from $H^0(X_2,L_2)$ according to our assumption. So
\begin{eqnarray*}
 \text{\text{\text{\text{dim}}}}~p_1(Y) & \leq & \text{\text{\text{\text{dim}}}}~(Y) -2 \\
                                        & = &  \text{\text{\text{\text{dim}}}}~(H^{1}(X_2,L_2^{*})) - 1.
\end{eqnarray*}
This implies that there exists an extension $e_2^{\prime} \in H^{1}(X_2,L_2^{*})$ which is not in $p_1(Y).$ So if $E_2^{\prime}$ is the bundle corresponding to $e_2^{\prime},$ then by equation \eqref{16}, $\text{\text{\text{\text{dim}}}}~(H^{0}(X_2,E_2^{\prime}))=1.$ But this is a contradiction as
$\text{\text{\text{\text{dim}}}}~(H^{0}(X_2,E_2)) \geq 2$ for every extension in $H^1(X_2,L_2^{*}).$

This proves that there exists a non-zero section $\phi_2 \in H^0(X_2,L_2)$ and an extension $e_2 \in \text{\text{\text{\text{ker}}}}~\tilde{\phi_2}$ such that $\phi_2$ is the only non-zero section (up to scalar multiplication) which lifts to the corresponding bundle $E_2.$ So by equation \eqref{16}, $\text{\text{\text{\text{dim}}}}~(H^{0}(X_2,E_2))=2.$

Now to prove that such an $E_2$ is semi-stable, let $G_2$ be a line sub-bundle of $E_2.$ We want to prove  $\text{\text{\text{\text{deg}}}}~(G_2) \leq \frac{\text{\text{\text{\text{deg}}}}~(E_2)}{2}=\frac{2g_2}{2}=g_2.$ Suppose $\text{\text{\text{\text{deg}}}}~(G_2) > g_2,$ then $\chi(G_2) > 1$ and  $\text{\text{\text{\text{dim}}}}~H^0(X_2,G_2) > 1.$ But $\text{\text{\text{\text{dim}}}}~(H^{0}(X_2,E_2))=2$ and $G_2 \subset E_2.$ So $\text{\text{\text{\text{dim}}}}~H^0(X_2,G_2)=2.$ This implies the map $\mathcal{O}_{X_2} \rightarrow E_2$ in the extension \eqref{13} factors through $G_2.$ This forces $G_2$ to be isomorphic to $\mathcal{O}_{X_2}.$ This implies $\text{\text{\text{\text{deg}}}}~(G_2)=0,$ which contradicts the fact that $\text{\text{\text{\text{deg}}}}~(G_2) >g_2.$

This proves that $E_2$ is semi-stable.
\end{proof}

Now by \cite[Proposition 2.2]{N-S}, we have the following exact sequence-
\begin{equation*}
 0\rightarrow L^* \rightarrow i_{1_{*}}(L_{1}^*) \oplus i_{2_{*}}(L_{2}^*) \rightarrow T_{\lambda^*} \rightarrow 0,
\end{equation*}
where $(L_1^*,L_2^*,\overrightarrow{\lambda^*})$ is the triple corresponding to $L^*.$
By taking the long exact sequence corresponding to this and observing that $H^0(X,L^*)=H^0(X_1,L_1^*)=H^0(X_2,L_2^*)=0,$ we get -
\begin{equation}\label{Equation 3.12}
 0\rightarrow H^{0}(X,T_{\lambda^*}) \rightarrow H^{1}(X,L^*) \xrightarrow{\gamma}  H^{1}(X_{1},L_{1}^*) \oplus H^{1}(X_{2},L_{2}^*) \rightarrow 0.
\end{equation}
Let $e_1 \in H^1(X_1,L_1^*)$ be an extension as in Lemma \ref{Lemma 3.5} and $e_2 \in H^1(X_2,L_2^*)$ be an extension as in Lemma \ref{Lemma 3.6}. Let the corresponding extensions be \eqref{12} and \eqref{13} respectively. Then $E_1$ is stable, and $E_2$ is semi-stable. Since the map $\gamma$ in the exact sequence \eqref{Equation 3.12} is surjective, given $(e_1,e_2) \in H^{1}(X_{1},L_{1}^*) \oplus H^{1}(X_{2},L_{2}^*),$ there exists an $e \in H^{1}(X,L^*)$ such that $\gamma(e)=(e_1,e_2).$ Let
\begin{equation}\label{Equation 3.13}
 0\rightarrow \mathcal{O}_{X} \rightarrow E \rightarrow L \rightarrow 0
\end{equation}
be the extension corresponding to $e.$ Then $E|_{X_1}=E_1,$  $E|_{X_2}=E_2,$ and so the triple corresponding to $E$ will look like $(E_1,E_2,\overrightarrow{T}),$ where $T:E_1(p) \rightarrow E_2(p)$ is an isomorphism at the node $p.$ Since $E_1$ and $E_2$ are semi-stable and $T$ has full rank, by \cite[Lemma 2.3]{S-B}, the triple $((E_1,E_2,\overrightarrow{T}))$ is semi-stable. So the corresponding vector bundle $E$ is semi-stable. Since $\chi(E)=1$ and $a_1\chi$ is not an integer, semi-stability coincides with stability. So $E$ is stable. Thus we have produced an extension of the form \eqref{Equation 3.13} in $H^1(X,L^*)$ such that $E$ is stable.

\begin{remark}\label{Remark 3.7}
 From the above arguments, it is clear that there exists an extension in $H^1(X,L^*)$ of the form \eqref{Equation 3.13} such that the corresponding bundle $E$ is stable. Since stability is an open condition, the set $B = \{v \in H^{1}(X,L^{*}) ~|~ \tilde{\mathcal{E}}_{v}~is~stable\}$ is a non-empty $k^*-$ invariant open set in $H^1(X,L^*),$ where $\tilde{\mathcal{E}}$ is as in Lemma \ref{Lemma 3.3}. Since $H^1(X,L^*)$ is irreducible, $B \cap W$ is a non-empty $k^*-$ invariant open set in $H^1(X,L^*),$ where $W$ is as in Remark \ref{Remark 3.2}(b). This implies $B \cap A \neq \emptyset,$ where $A$ is as defined in Remark \ref{Remark 3.2}(b). Let $S= B \cap A.$ Then $S$ is a non-empty open subset of the affine space $A$ consisting of stable rank two locally free sheaves $\tilde{\mathcal{E}}_{s}.$
\end{remark}

\section{Rationality}
 We are now in a position to state and prove the main proposition -
 \begin{proposition}\label{Proposition 4.1}
 Let $\chi$ be an odd integer and $a_{1},$ $a_{2}$ be rational numbers such that $0 < a_{1}< a_{2} < 1$ and $a_{1}+a_{2}=1.$ Suppose $a_{1}\chi$ is not an integer and $\chi_{1}$ and $\chi_{2}$ are integers such that
 \begin{eqnarray}\label{17}
  a_{1}\chi < \chi_{1} < a_{1}\chi +1 , ~a_{2}\chi +1 < \chi_{2} < a_{2}\chi +2.
 \end{eqnarray}
  Let $L=(L_{1},L_{2},\overrightarrow{\lambda})$ be an invertible sheaf on $X$ such that $\text{\text{deg}}~(L_{1})= \chi_{1}-2(1-g_{1})$ and $\text{\text{deg}}~(L_{2})=\chi_{2}-2(1-g_{2}).$ Then there exists a non-empty open subset $S$ of an affine space and a locally free sheaf $\mathcal{E}$ of rank two on $S \times X$ such that

  (i) $\text{\text{\text{\text{dim}}}}~(S)=3g-3,$

  (ii) $\mathcal{E}_{s} \cong \mathcal{E}_{t}$  $\Leftrightarrow$  $s=t,$

  (iii) for all $s \in S,$ $\mathcal{E}_{s}$ is stable and $\Lambda^{2}(\mathcal{E}_{s})=L.$
 \begin{proof}
 We prove the Proposition by considering the following two different cases.

 $\mathbf{Case~ 1:}$ Suppose $\chi_{1}$ is odd and $\chi_{2}$ is even and $\chi_1$, $\chi_2$ satisfies \eqref{17}.
 In order to prove the Proposition for this case, we first assume $\chi=1$. This implies $\chi_{1}=1,$ $\chi_{2}=2,$ $\text{\text{deg}}~(L_{1})=2g_{1}-1$, $\text{\text{deg}}~(L_{2})=2g_{2}$ and $L_{2}$ is globally generated. We further assume $L_1$ is globally generated and prove the Proposition in this case. Given such an $L,$ we can apply Lemma \ref{Lemma 3.3} and get an extension $\tilde{\mathcal{E}}$ on $H^{1}(X,L^{*}) \times X$ as in \eqref{11}. Let $q:=\text{\text{\text{\text{dim}}}}~(H^{1}(X,L^{*}))=(3g-2).$ From Remark \ref{Remark 3.2}(b), we obtain an affine space $A$ which is closed in $H^{1}(X,L^{*})$ and is of \text{\text{\text{\text{dim}}}}ension $3g-3.$ Let $S$ be the open subset of $A$ defined in Remark \ref{Remark 3.7}.

 Let $\mathcal{E}^{\prime} = \tilde{\mathcal{E}}\mid_{S \times X}.$ By Remark \ref{Remark 3.7}, $\mathcal{E}^{\prime}_{s}$ is stable and $\Lambda^{2}(\mathcal{E}^{\prime}_{s})=L,$ $\forall~s \in S.$ Since $S$ is open in $A,$ $\text{\text{\text{\text{dim}}}}~(S)=3g-3.$ By the choice of $A,$ and Remark \ref{Remark 3.4}, it is clear that $\mathcal{E}^{\prime}_{s} \cong \mathcal{E}^{\prime}_{t}$ $\Leftrightarrow$ $s=t.$ This proves the proposition for the case $\chi=1$ and $L_{1}$ globally generated.

 Now to prove the proposition for arbitrary odd $\chi_1$ and even $\chi_2$, let $\text{\text{deg}}~(L_{1})=2l_{1}-1,$ and $\text{\text{deg}}~(L_{2})=2l_{2},$ where $l_{1},l_{2}$ are integers. Let $M$ be an invertible sheaf on $X$ given by the triple $(M_{1},M_{2},\overrightarrow{\delta})$ such that $\text{\text{deg}}~(M_{1})=l_{1}-g_{1},$  $\text{\text{deg}}~(M_{2})=l_{2}-g_{2}.$
 Since the pull-back operation commutes with tensor product, the triple corresponding to the invertible sheaf $L \otimes M^{-2}$ will be $(L_{1}\otimes M_{1}^{-2}, L_{2} \otimes M_{2}^{-2}, \overrightarrow{\gamma}),$ where $\overrightarrow{\gamma}$ is the corresponding map at $p$ induced by $\overrightarrow{\lambda}$ and $\overrightarrow{\delta}.$ Clearly $\text{\text{deg}}~(L_{1} \otimes M_{1}^{-2})=2g_{1}-1,$ and $\text{\text{deg}}~(L_{2} \otimes M_{2}^{-2})=2g_{2}.$ We can also assume that $L_{i} \otimes M_{i}^{-2}$ is generated by global sections for $i=1,2,$ by choosing an appropriate $M_{1}$ (see Remark \ref{Remark 4.2}).

 Now corresponding to this invertible sheaf $L \otimes M^{-2},$ we have proved above that there exists a locally free sheaf $\mathcal{E}^{\prime}$ on $S \times X$ satisfying all the required properties. Let  $\mathcal{E}:= \mathcal{E}^{\prime} \otimes p_{X}^{*}(M).$ So $\mathcal{E}_{s} = \mathcal{E}^{\prime}_{s} \otimes M.$ Clearly $\Lambda^{2}(\mathcal{E}_{s})=L$ for every $s \in S.$

 We now claim that $\mathcal{E}_{s}$ is stable. Let $\mathcal{E}_{s}^{\prime}=(E_{1},E_{2},\overrightarrow{T}).$ Then it is clear that $\chi(E_{1})=1$ and $\chi(E_{2})=2.$ Also it is clear that the triple corresponding to $\mathcal{E}_{s}$ is $(E_{1} \otimes M_{1}, E_{2} \otimes M_{2}, \overrightarrow{T} \otimes \overrightarrow{\delta}).$ Now since $\mathcal{E}_{s}^{\prime}$ is stable, by \cite[Theorem 5.1]{N-S}, $E_{1}$ and $E_{2}$ are semi-stable. So $E_{j} \otimes M_{j}$ is semi-stable, for $j=1,2.$ Since $\overrightarrow{T} \otimes \overrightarrow{\delta}$ has full rank, by \cite[Lemma 2.3]{S-B},
\begin{eqnarray*}
 semi-stability~ of~ E_{j} \otimes M_{j} & \Rightarrow & \mathcal{E}_{s} ~ is~semi-stable.
\end{eqnarray*}
Since $\chi(\mathcal{E}_{s})$ is odd and $a_{1}\chi$ is not an integer, semi-stability of $\mathcal{E}_{s}$ implies it is stable for every $s \in S.$

 $\mathbf{Case~ 2:}$ Suppose $\chi_{1}$ is even and $\chi_{2}$ is odd. Let $L_{1}^{\prime}$ be an invertible sheaf on $X_{1}$ of \text{\text{deg}}ree $2g_{1}$ and $L_{2}^{\prime}$ be an invertible sheaf on $X_{2}$ of \text{\text{deg}}ree $2g_{2}-1$ such that it is globally generated. Let $L^{\prime}=(L_{1}^{\prime},L_{2}^{\prime},\overrightarrow{\lambda^{\prime}}),$ where $\lambda^{\prime}$ is a non-zero scalar. Then one can prove all results of section (3) by replacing $L$ in those results by $L^{\prime}$ (The proofs are similar). \\
 Now to prove the Proposition in this case, one first proves the existence of $S$ and a locally free sheaf $\mathcal{E}^{\prime}$ of rank two on $S \times X$ such that $\Lambda^{2}(\mathcal{E^{\prime}}_{s})=L^{\prime},~\forall s \in S$  as in  Case (1), and then tensoring $\mathcal{E}^{\prime}$ by a suitable invertible sheaf $M'\,=\,(M_1',M_2',\overrightarrow{\delta)}$ as in Case(1), one gets a locally free sheaf $\mathcal{E}$ of rank two on $S \times X$ satisfying all the required properties.
 \end{proof}
 \end{proposition}

 \begin{remark}\label{Remark 4.2}
  Let $L_{1}$ be an invertible sheaf on $X_{1}$ of \text{\text{deg}}ree $2l_{1}-1,$ where $l_{1}$ is any integer. Then the invertible sheaf $\omega_{X_{1}} \otimes \mathcal{L}_{1} \otimes L_{1}^{-1}$ (where $\omega_{X_{1}}$ and $\mathcal{L}_{1}$ are as in Remark \ref{Remark 3.2}(a), is of \text{\text{deg}}ree $2(g_{1}-l_{1}).$ So there exists an invertible sheaf $N_{1}$ of \text{\text{deg}}ree $(g_{1}-l_{1})$ such that $N_{1}^{2}= \omega_{X_{1}} \otimes \mathcal{L}_{1} \otimes L_{1}^{-1}.$ Let $M_{1}=N_{1}^{-1}.$ Then $L_{1} \otimes M_{1}^{-2} = \omega_{X_{1}} \otimes \mathcal{L}_{1}$ which is globally generated and is of \text{\text{deg}}ree $2g_{1}-1.$
 \end{remark}

\subsection{Proof of Theorem 1.1}

Let $\chi$ and $L$ be as in Proposition \ref{Proposition 4.1}. Then there exists a non-empty open subset $S$ of an affine space and a rank two locally free sheaf $\mathcal{E}$ on $S \times X$ such that properties $(i),$ $(ii)$ and $(iii)$ of Proposition \ref{Proposition 4.1} are satisfied. Since $M_{12}({\bf \xi})$ is a fine moduli space, the sheaf $\mathcal{E}$ on $S \times X$ induces a morphism
 \begin{equation*}
  f: S \rightarrow M_{12}({\bf \xi}).
 \end{equation*}
 By $(ii)$ of Proposition \ref{Proposition 4.1}, $f$ is injective. Since $S$ and $M_{12}({\bf \xi})$ are of the same \text{\text{\text{\text{dim}}}}ension and we are in characteristic zero, this implies that $f$ is birational. So $M_{12}({\bf \xi})$ is rational.

Now, we briefly outline the proof of rationality of the other component $M_{21}({\bf \xi}).$

Suppose $L_{1}$ and $L_{2}$ are as in Proposition \ref{Proposition 4.1}. Then clearly $\text{\text{deg}}~(L_{1} \otimes \mathcal{O}_{X_{1}}(p)) = \chi_{1}^{\prime} - 2(1-g_{1})$ and $\text{\text{deg}}~(L_{2} \otimes \mathcal{O}_{X_{2}}(-p)) = \chi_{2}^{\prime} - 2(1-g_{2}),$ where $\chi_{1}^{\prime}$ and $\chi_{2}^{\prime}$ are integers satisfying the inequalities \begin{equation}\label{18}
  a_{1}\chi + 1 < \chi_{1}^{\prime} < a_{1}\chi +2 , ~a_{2}\chi  < \chi_{2}^{\prime} < a_{2}\chi +1.
 \end{equation}
 Let $\hat{L}=(L_{1} \otimes \mathcal{O}_{X_{1}}(p),L_{2} \otimes \mathcal{O}_{X_{2}}(-p),\overleftarrow{\lambda}),$ where $\lambda$ is a non-zero scalar. Then one can show exactly as in Proposition \ref{Proposition 4.1} that there exists a non-empty open subset $\hat{S}$ of an affine space and a rank two locally free sheaf $\mathcal{\hat{E}}$ on $\hat{S} \times X$ such that

  (i) $\text{\text{\text{\text{dim}}}}~(\hat{S})=3g-3,$

  (ii) $\mathcal{\hat{E}}_{s} \cong \mathcal{\hat{E}}_{t}$  $\Leftrightarrow$  $s=t,$

  (iii) for all $s \in \hat{S},$ $\mathcal{\hat{E}}_{s}$ is stable and $\Lambda^{2}(\mathcal{\hat{E}}_{s})=\hat{L}.$

 Now since $M_{21}({\bf \xi})$ is a fine moduli space, the sheaf $\mathcal{\hat{E}}$ on $\hat{S} \times X$ induces a morphism
 \begin{equation*}
  g: \hat{S} \rightarrow M_{21}({\bf \xi}).
 \end{equation*}
 Also the fact that $\mathcal{\hat{E}}_{s} \cong \mathcal{\hat{E}}_{t}$  $\Leftrightarrow$  $s=t,$ implies $g$ is injective. Since $\hat{S}$ and $M_{21}({\bf \xi})$ are of same \text{\text{\text{\text{dim}}}}ension and we are in characteristic zero, it implies that $g$ is birational. So $M_{21}({\bf \xi})$ is rational. This proves Theorem \ref{Theorem 1.1}.

\end{document}